\documentclass[12pt,sn-mathphys-ay]{sn-jnl}

\usepackage{graphicx}%
\usepackage{multirow}%
\usepackage{amsmath,amssymb,amsfonts,}%
\usepackage{mathtools}
\usepackage{amsthm}%
\usepackage{mathrsfs}%
\usepackage[title]{appendix}%
\usepackage{xcolor}%
\usepackage{textcomp}%
\usepackage{manyfoot}%
\usepackage{booktabs}%
\usepackage{algorithm}%
\usepackage{algorithmicx}%
\usepackage{algpseudocode}%
\usepackage{listings}%
\usepackage{hyperref}
\usepackage{bm}
\usepackage{dsfont}
\usepackage[labelfont=bf,labelsep=space]{caption}
\usepackage{enumitem}
\theoremstyle{thmstyleone}%
\newtheorem{theorem}{Theorem}
\newtheorem{theoremletter}{Theorem}

\newtheorem{lemmaletter}[theoremletter]{Lemma}

\theoremstyle{thmstyletwo}%
\newtheorem{remark}{Remark}%

\theoremstyle{thmstylethree}%
\begin{document}
\newcommand{\R}{\mathds{R}}
\newcommand{\So}{\mathcal{S}}
\newcommand{\M}{M^{+}(\R^n)}
\newcommand{\wpp}{\mathbf{H}_{0}^{1}(\Omega)}
\newcommand{\woo}{\mathbf{W}_{0}^{1,1}(\Omega)}
\newcommand{\wa}{\mathbf{W}_{\alpha,p}}
\newcommand{\wao}{\mathbf{W}_{\alpha,p}}
\newcommand{\wat}{\mathbf{W}_{\alpha,p}}
\newcommand{\wak}{\mathbf{W}_{k}}
\newcommand{\wah}{\mathbf{W}_{\frac{2k}{k+1},k+1}}
\newcommand{\ia}{\mathbf{I}_{2\alpha}}
\newcommand{\Om}{\Omega}
\newcommand{\e}{\epsilon}
\newcommand{\g}{\gamma}
\newcommand{\al}{\alpha}
\newcommand{\te}{\theta}
\newcommand{\un}{u_{n}}
\newcommand{\vf}{\varphi}
\newcommand{\wc}{\rightharpoonup}
\newcommand{\phik}{\Phi^{k}(\Om)}
\newcommand{\la}{\lambda}
\newcommand{\dx}{\mathrm{dx}}
\newcommand{\Div}{\mathrm{div}}
\newcommand{\Linf}{L^{\infty}(\Om)}
\newcommand{\iom}{\int_{\Omega}}
\newcommand{\ioa}{\int_{A_{k}}}
\def\norma#1#2{\|#1\|_{\lower 4pt \hbox{$\scriptstyle #2$}}}
\def\L#1{L^{#1}(\Omega)}

\title{An elliptic equation with power nonlinearity and degenerate coercivity}
\author{\fnm{Genival} \sur{da Silva}\footnote{email: gdasilva@tamusa.edu, website: \url{www.gdasilvajr.com}}}
\affil{\orgdiv{Department of Mathematics}, \orgname{Texas A\&M University - San Antonio}}


\abstract{
We discuss the existence and regularity of solutions to the following Dirichlet problem: 
\begin{equation} 
\begin{cases}
-\Div\left(\frac{Du}{(1+|u|)^{\te}}\right)= -\Div\left(u^{\g}E(x)\right)+f(x) \qquad & \mbox{in } \Omega,\\
u (x) = 0 & \mbox{on }  \partial \Omega,
\end{cases}
\end{equation}
where $\te,\g>0$. An interesting feature of this problem is the interplay between the two nonlinearities, the degeneracy and the power nonlinearity.
}




\maketitle
\section{Introduction}
In these notes we study existence and regularity of solutions to a class of elliptic problems whose basic model is
\begin{equation} 
\begin{cases}
 -\Div\left(\frac{Du}{(1+|u|)^{\te}}\right)= -\Div\left(u^{\g}E(x)\right)+f(x) \qquad & \mbox{in } \Omega,\\
u (x) = 0 & \mbox{on }  \partial \Omega,
\end{cases}
\end{equation}
where $\te>0$, $E(x)$ is a vector field and $f(x)$ a function in $\L m$ with $m\ge 1$.

More generally, we will focus on the following problem:
\begin{equation} \label{main}\tag{D}
\begin{cases}
 -\Div(a(x,u)Du)= -\Div\left(u^{\g}E(x)\right)+f(x) \qquad & \mbox{in } \Omega,\\
u (x) = 0 & \mbox{on }  \partial \Omega,
\end{cases}
\end{equation}
where $\g$ is a fixed number, $a(x,s)$ is a Caratheodory function which satisfies, for a.e. $x\in\Om$, any $s\in\R$:
\begin{equation}\label{degc}\tag{C}
\frac{\al}{(1+|s|)^{\te}}\le a(x,s)\le \beta,
\end{equation}
where $\al,\beta$ are positive constants. 

The summability of $E(x)$ and $f(x)$ will vary and will be specified below.
\vspace{0.1in}
\begin{remark}
Condition \eqref{degc} implies that the operator in problem \eqref{main} is not coercive, hence the usual $\wpp$-existence theory cannot be applied directly. Moreover, since $E(x)$ is not necessarily a potential, i.e. $E=D v$, the equation is not always variational.
\end{remark}
\vspace{0.1in}
Most of our results will assume $\te+\g$ small, the case where $\g$ is large and $\te=0$ has been recently described in \cite{boc24}. If $\te>1$, some nonexistence results exist, see \cite{boc03}. The rationale behind all these results is that when the summability of $E(x)$ and $f(x)$ is high enough, bounded weak solutions tend to exist, whereas in low summability cases we can only get distributional solution lying in some Sobolev space $\textbf{W}_{0}^{1,q}(\Om)$, for some $1<q<2$, and in some cases, a smallness condition on the source $f(x)$ is also required.

We will look for two types of solutions:

A \textbf{weak solution}, sometimes also called a \textit{finite energy} solution, is a function $u\in\wpp$ such that for $f\in \L{2_{*}}$, $u^{\g} E\in [\L 2]^{n}$ and we have:
\begin{equation}\label{weak}
\int_{\Omega} a(x,u)Du  D \varphi =\int_{\Omega}u^{\g}E(x)  D \varphi  + \int_{\Omega} f \varphi \quad \forall \varphi \in  \wpp.
\end{equation}
\vspace{0.1in}
A \textbf{distributional solution} is a function $u\in\woo$ such that $f\in \L 1$, $u^{\g} E\in [L^{1}_{loc}(\Om)]^{n}$ and we have:
\begin{equation}\label{dis}
\int_{\Omega} a(x,u)Du  D \varphi =\int_{\Omega}u^{\g}E(x)  D \varphi  + \int_{\Omega} f \varphi \quad \forall \varphi \in  \mathcal{C}^{\infty}_{0}(\Om).
\end{equation}
\vspace{0.1in}
\begin{remark}
Notice that every weak solution $u\in\wpp$, if there is one, is a distributional solution by definition.
\end{remark}
\vspace{0.1in}
As mentioned above, nonexistence can occur when $\te>1$. As a counterweight, in the last section of this paper we add a lower order term to problem \eqref{main} and are able to recover existence and regularity of solutions in this scenario as well.

Existence and regularity of solutions to quasilinear elliptic equations is an old and interesting problem. The literature is vast, specially for semilinear equations, see for example \cite{ding86,brezis92,diaz93,djairo95,blions,ambrosetti,boc09,boc24}, and for more comprehensive treatment see \cite{trudinger, boc13}.

\subsection{Summary of results}
For the convenience of the reader we present a brief summary of what will be proved in this paper. We start with problem \eqref{main} and assume $0<\te,\g<1$, $f\in \L{m}$, $E\in [\L{2r}]^{n}$. The following will be proved:
\begin{itemize}
	\item If $\te+\g<\frac{2^{*}}{2}$ and $m,r$ are sufficiently large then there is a bounded weak solution $u\in\wpp\cap\L\infty$.
	\item If $\te+2\g<1$, we drop the high summability of $E(x)$, we get a distributional solution $u\in\textbf{W}_{0}^{1,q}$, where $1<q<2$.
	\item if $\te+\g<\frac 1 2$ and we consider $f\in\L1$ only, we still have a distributional solution $u\in\textbf{W}_{0}^{1,q}$, with $1<q<2$.
\end{itemize}
In the last section, we study the problem \eqref{main2} with $\te>0$ but still maintaining $\g$ small. We prove:
\begin{itemize}
	\item If $\te+2\g<2$, $m$ and $r$ sufficiently large then there is a bounded weak solution $u\in\wpp\cap\L\infty$.
	\item If $\te+2\g<2$, $m\ge \te+2$ and $r$ sufficiently large then there is a distributional solution $u\in\wpp\cap\L m$.
	\item If $\te+2\g<2$, $2 \le m< \te+2$ and $r$ sufficiently large then there is has distributional solution $u\in\wpp\cap\L m$.
\end{itemize}
\subsection{Notation \& Assumptions}
\begin{itemize}
\item[-] $\Om\subset \R^{N}$ is a bounded domain and $N\ge 3$.
\item[-] The space $\wpp$ denotes the usual Sobolev space which is the closure of $\mathcal{C}^{\infty}_{0}(\Om)$, smooth functions with compact support using the Sobolev norm. 
\item[-] For $q>1$, $q'$ denotes the Holder conjugate, i.e. $\frac{1}{q}+\frac{1}{q'}=1$, and $q^{*}$ denotes the Sobolev conjugate, defined by $q^{*}=\frac{qN}{N-q}>q$. 
\item[-] For $p>1$, $p_{*}$ denotes $(p^{*})'$, in particular, $2_{*}=\frac{2N}{N+2}$.
\item[-] We will use the somewhat standard notation for Stampacchia's truncation functions (see \cite{boc13}):
\[
T_k(s)=\max\{-k,\min\{s,k\}\}, \ \ \ G_k(s)=s-T_k(s), \quad \mbox{ for } k>0.
\]
\item[-] The letter $C$ will always denote a positive constant which may vary from place to place.
\item[-] The Lebesgue measure of a set $A\subseteq \R^{n}$ is denoted by $|A|$.
\item[-] The symbol $\rightharpoonup$ denotes weak convergence. 
\end{itemize}
\section{Proof of the results}
Fix $n>0$, let $f_{n}(x)=T_n(f(x))$ and $E_{n}(x)=T_n(E(x))$, the latter is the vector field  obtained from $E(x)$ by truncating its components by $n$. Consider the truncated equation:
\[
 -\Div(a(x,T_{n}(u_{n}))D\un)= -\Div\left(T_{n}(\un^{\g})E_{n}\right)+f_{n}(x)
\]
A simple application of Schauder’s fixed point theorem guarantee the existence of weak solution, i.e. a function $u_{n}\in\wpp$ satisfying
\begin{equation}\label{wkf}
\int_{\Omega} a(x,T_{n}(u_{n}))Du_{n} D \varphi =\int_{\Omega}\left(T_{n}(\un^{\g})E_{n}\right) D \varphi  + \int_{\Omega} f_{n} \varphi \quad \forall \varphi \in  \wpp.
\end{equation}
Moreover, since the right hand side is bounded, classical regularity results imply that $\un\in\L\infty$ as well.

The following lemma will be needed below.
\vspace{0.1in}
\begin{lemmaletter}\label{lemA}\cite[Lem~6.2]{boc13}
Let $f\in L^{m}$, $m>\frac N 2$, $g(k)=\int_{\Om} |G_{k}(f)|$ and $A_{k}=\{ |f|>k \}$. Suppose
\[
g(k)\le \beta |A_{k}|^{\al}
\]
for some $\al>1$ and $\beta>0$. Then $f\in \L{\infty}$ and 
\[
\norma{f}{\infty}\le C\beta 
\]
for some $C=C(\al,\Om)$. 
\end{lemmaletter}
\subsection{When $m,r$ are sufficiently large}
In our first result below, we seek \textit{finite energy solutions}, that is, bounded weak solutions $u\in\wpp$. As mentioned in the introduction, the majority of results of this type require high summability on the source. The theorem below confirms that claim and also requires an additional summability of the vector field $E(x)$ as well. 
\vspace{0.1in}
\begin{theorem} \label{bounds}
Suppose \eqref{degc} holds with $0<\te<2^{*}-2$ and $0<\g<1$ satisfying
\[
\te+\g<\frac{2^{*}}{2},
\]
$E\in [\L {2r}]^{N}$, $f\in\L m$ satisfying
\begin{equation}\label{eq11}
\begin{split}
r&>\frac{(\g+2)}{(\g+1)}\frac{N}{2},\\ m&>\frac{N}{2},\\ 
\end{split}
\end{equation}
Then the Dirichlet problem \eqref{main} has a weak bounded solution $u\in\wpp\cap\L \infty$.
\end{theorem}
\begin{proof}
Consider $\vf= [(1+|\un|)^{\te+1}-1]\textrm{sgn}(\un)$ as a test function in \eqref{wkf}. We have
\[
\alpha (\te+1)\iom |D \un|^2\le \iom (1+|\un|)^{\te+\g}|E||D \un|+\iom |f|(1+|\un|)^{\te+1}.
\] 
Applying Young's inequality to the above:
\begin{equation}\label{eq3}
\frac{\al(\te+1)}{2} \iom |D \un|^2\le \frac{1}{2\al(\te+1)}\iom (1+|\un|)^{2(\te+\g)}|E|^{2}+\iom |f|(1+|\un|)^{\te+1},
\end{equation}
using Sobolev's and Holder's inequalities:
\begin{equation}
\begin{split}
\frac{\al(\te+1)}{2} \left(\iom (1+ |\un|)^{2^{*}}\right)^{\frac{2}{2^{*}}} &\le \frac{1}{2\al(\te+1)}\left(\iom |E|^{2r}\right)^{\frac{1}{r}}\left(\iom (1+|\un|)^{\frac{2r(\te+\g)}{r-1}}\right)^{\frac{r-1}{r}}\\&+\left(\iom |f|^{m}\right)^{\frac{1}{m}}\left(\iom (1+|\un|)^{\frac{m(\te+1)}{m-1}}\right)^{\frac{m-1}{m}}
\end{split}
\end{equation}
Choosing $r,m$ such that $2^{*}=\frac{2r(\te+\g)}{r-1}=\frac{m(\te+1)}{m-1}$, which is to say:
\[
r=\frac{2^{*}}{2^{*}-2(\te+\g)}\text{ and } m=\frac{1}{2^{*}-\te-1}
\]
Simplifying we have:
\[
\left(\iom  |\un|^{2^{*}}\right)^{\frac{2}{2^{*}}-\frac{1}{m'}}\le C(\norma{E}{2r}^{2} + \norma{f}{m})
\]
Which implies by \eqref{eq3} that:
\[
\norma{D\un}{2}\le C.
\]
Now we prove that $\un$ is bounded in $\L\infty$. Define
\[
H(s)=\int_{0}^{s}\frac{1}{(1+|s|)^{\te}}.
\]
Set $A_{k}=\{ x\in\Om\,|\, H(\un)>k \}$, taking $\vf= G_{k}(H(\un))$ as a test function in \eqref{wkf} we obtain
\[
\al \int_{A_{k}} |DH(\un)|^2\le \int_{A_{k}} |\un|^{\g}|E||D H(\un)|+\int_{A_{k}} |f|G_{k}(H(\un))
\]
Simplifying using Holder's inequality:
\[
\al \int_{A_{k}} |DH(\un)|^2\le \left(\int_{A_{k}} |\un|^{2\g}|E|^{2}\right)^{\frac12}\left(\int_{A_{k}} |DH(\un)|^{2}\right)^{\frac12}+\left(\int_{A_{k}} |f|^{2_{*}}\right)^{\frac{1}{2_{*}}}\left(\int_{A_{k}} |DH(\un)|^{2}\right)^{\frac12}
\]
We conclude that
\[
\al\left(\int_{A_{k}} |DH(\un)|^{2}\right)^{\frac12}\le \left(\int_{A_{k}} |\un|^{2\g}|E|^{2}\right)^{\frac12}+\left(\int_{A_{k}} |f|^{2_{*}}\right)^{\frac{1}{2_{*}}}
\]
Using Holder's inequality again:
\[
 \al\left(\int_{A_{k}} |DH(\un)|^{2}\right)^{\frac12}\le \left(\int_{A_{k}} |\un|^{2^{*}}\right)^{\frac{\g}{2^{*}}}\left(\int_{A_{k}} |E|^{\frac{22^{*}}{2^{*}-2\g}}\right)^{\frac{2^{*}-2\g}{22^{*}}}+\norma{f}{m}|A_{k}|^{\frac{m-2_{*}}{2_{*}m}},
\]
Now, by Sobolev's inequality:
\[
\left(\int_{A_{k}} |G_{k}(H(\un))|^{2^{*}}\right)^{\frac{1}{2^{*}}}\le C\left(\norma{E}{r}|A_{k}|^{\frac{r-\frac{22^{*}}{2^{*}-2\g}}{r}\frac{2^{*}-2\g}{22^{*}}}+\norma{f}{m}|A_{k}|^{\frac{m-2_{*}}{2_{*}m}}\right),
\]
Finally, recall that:
\[
\int_{A_{k}} |G_{k}(H(\un))|\le \left(\int_{A_{k}} |G_{k}(H(\un))|^{2^{*}}\right)^{\frac{1}{2^{*}}}|A_{k}|^{\frac{1}{2_{*}}}
\]
Combing this with the condition \eqref{eq11}, we conclude that by Lemma \ref{lemA}, $\norma{H(\un)}{\infty}\le C$, but since $\lim_{s\to\pm\infty}H(s)=\pm\infty$, we deduce
\[
\norma{\un}{\infty}\le C.
\]
Choose $n>C$ with $T_{n}(\un)=\un$, then $\un$ is a weak solution.
\end{proof}
\vspace{0.1in}
\begin{remark}\label{rm2}
This theorem reinforces what happens in the case $E=0$. Since in that case it's possible to obtain nonexistence results if $\te>1$ and existence only if the source is small.

For example, if we take $N=3, \te=2, E=0$ and a constant $A>0$ large enough then the problem
\begin{equation} 
\begin{cases}
-\Div\left(\frac{Du}{(1+|u|)^{\te}}\right)= A\qquad & \mbox{in } B_{1}(0),\\
u (x) = 0 & \mbox{on }  \partial B_{1}(0),
\end{cases}
\end{equation}
doesn't have a weak solution $u\in\wpp$, see  \cite[Sec.~14.3]{boc13}. Whereas using the theorem above and noting that $\te<1$, we have bounded weak solutions without smallness condition on the source $f(x)$.
\end{remark}
\subsection{Low summability of $E(x)$, i.e. $E\in [\L 2]^{N}$}
In our next result, we drop the summability assumption on $E(x)$ and as a result, finite energy solutions are not guaranteed to exist anymore and we can only hope for distributional solutions as the theorem below shows.
\vspace{0.1in}
\begin{theorem} \label{bounds2}
Suppose \eqref{degc} holds with $0<\te<1$ and  $0<\g<1$ satisfying
\[
\te+2\g<1,
\]
$E\in [\L 2]^{N}$, $f\in\L m$ with $m>\frac{q^{*}}{q^{*}-1+\te+2\g}$, where $q=\frac{2N (1-\te-\g)}{N-2(\te+\g)}$. Then the Dirichlet problem \eqref{main} has a distributional solution $u\in\textbf{W}_{0}^{1,q}$.
\end{theorem}
\begin{proof}
Set $\vf= [(1+|\un|)^{\la}-1]\textrm{sgn}(\un)$ as a test function in \eqref{wkf}, where $\la<1$ will be specified later. We have
\[
\alpha \la\iom (1+|\un|)^{\la-1-\te}|D \un|^2\le \iom(1+|\un|)^{\g+\la-1} |E||D \un|+\iom |f||\un|^{\la}
\] 
Using Young's inequality we obtain:
\[
\alpha \la\iom (1+|\un|)^{\la-1-\te}|D \un|^2\le \frac{1}{2\alpha \la}\iom |E|^{2}+\frac{\alpha \la}{2}\iom (1+|\un|)^{2(\g+\la-1)}|D \un|^{2}+\iom |f||\un|^{\la}
\]
Now choose $\la=1-\te-2\g$, simplifying we have:
\begin{equation}\label{eq1}
\frac{\alpha (1-\te-2\g)}{2}\iom \frac{|D \un|^2}{(1+|\un|)^{2(\te+\g)}}\le \frac{1}{2\alpha \la}\iom |E|^{2}+\iom |f||\un|^{\la}.
\end{equation}
For any $q<2$, by Holder's inequality using $\frac 2 q$ and $(\frac{2}{q})'$:
\begin{equation}\label{eq7}
\begin{split}
C \left(\iom |\un|^{q^{*}}\right)^{\frac{q}{q^{*}}}\le \iom |D \un|^q&\le \iom \frac{(1+|\un|)^{q(\te+\g)}|D \un|^q}{(1+|\un|)^{q(\te+\g)}}\\&\le \left(\iom \frac{|D \un|^2}{(1+|\un|)^{2(\te+\g)}}\right)^{\frac q 2} \left( \iom (1+|\un|)^{\frac{2q(\te+\g)}{2-q}}\right)^{\frac{2-q}{2}}
\end{split}
\end{equation}
Combining this with \eqref{eq1}:
\[
 \left(\iom |\un|^{q^{*}}\right)^{\frac{2}{q^{*}}}\le C\left( \iom (1+|\un|)^{\frac{2q(\te+\g)}{2-q}}\right)^{\frac{2-q}{q}}
\left[\frac{1}{2\alpha \la}\iom |E|^{2}+\iom |f||\un|^{\la}\right]
\]
Set $q=\frac{2N (1-\te-\g)}{N-2(\te+\g)}$, then $q^{*}=\frac{2q(\te+\g)}{2-q}$ and $\frac{2}{q^{*}}>\frac{2-q}{q}$. We have:
\[
 \left(\iom |\un|^{q^{*}}\right)^{\frac{2}{q^{*}}-\frac{2-q}{q}}\le C
\left[\iom |E|^{2}+\iom |f|^{m}+\iom |\un|^{\la m'}\right]
\]
Since $m$ satisfies $q^{*}=\la m'$ we obtain
\[
\norma{\un}{q^{*}}\leq C,
\]
Notice that by \eqref{eq7} we also have:
\[
\norma{D\un}{q}\leq C.
\]
It follows that $\un$ is bounded in $\textbf{W}_{0}^{1,q}$ and up to subsequence $\un\wc u$.

 By the dominated convergence theorem, we can easily pass the limit in the first integral in \eqref{wkf} if we assume $\vf\in\mathcal{C}^{\infty}_{0}(\Om)$. Similarly, we can pass the limit in the third integral, the only part not so obvious is the second integral. 

Notice that given $M\subset \Om$ measurable set:
\[
\int_{M}\left(T_{n}(\un)^{\g}E_{n}\right) D \varphi \le C\norma{D \varphi}{\infty}\norma{\un}{q^{*}}^{\g}\left(\int_{M} |E|^{\frac{q^{*}}{q^{*}-\g}}\right)^{\frac{q^{*}-\g}{q^{*}}}
\]
Therefore, the integral above is equi-integrable and the result follows from Vitali's convergence theorem.
\end{proof}
\subsection{$f(x)$ has low summability}
Now we assume that $f$ is integrable only and study the effect of this condition on the existence of solutions. As we should expect, low integrability of $f(x)$ implies that finite energy solutions don't exist, in fact, they are not even well defined. The theorem below is similar to the one before it, the major difference is the summability of $f(x)$, which in this case is kept to a minimum.
\vspace{0.1in}
\begin{theorem} \label{bounds3}
Suppose \eqref{degc} holds with $\te,\g>0$ satisfying
\[
\te+\g<\frac 1 2.
\]
Let $q=\frac{2N(1-(\te+\g))}{N-2(\te+\g)}$, $E\in [\L r]^{N}$, where $r>\frac{q^{*}}{\g}$, $f\in\L 1$. Then the Dirichlet problem \eqref{main} has a distributional solution $u\in\textbf{W}_{0}^{1,q}$.
\end{theorem}
\begin{proof}
The proof is similar to the one above. We begin by fixing $k>0$ and setting $\vf= T_{1}(G_{k}(\un))$ as a test function in \eqref{wkf}. We have
\[
\alpha \int_{B_{k}} |D \un|^2 \le (2+k)^{\te+\g}\int_{B_{k}}|E| |D \un|+(2+k)^{\te}\ioa |f|,
\] 
where $B_{k}=\{ x\in\Om\, | \,k\le |\un|< k+1 \}$ and $A_{k}=\{ x\in\Om\, | \,|\un|\ge k \}$ . Using Young's inequality we obtain:
\[
\alpha \int_{B_{k}} |D \un|^2 \le \frac{(2+k)^{2(\te+\g)}}{2\al}\int_{B_{k}}|E|^{2}+\frac{\al}{2}\int_{B_{k}}|D \un|^{2}+(2+k)^{\te}\left(\int_{B_{k}} |f| + |\{ |\un| \ge k+1\}|\right)
\]
Simplifying we have:
\[
\frac{\al}{2} \int_{B_{k}} |D \un|^2 \le \frac{(2+k)^{2(\te+\g)}}{2\al}\int_{B_{k}}|E|^{2}+(2+k)^{\te}\left(\int_{B_{k}} |f| + C\right)
\]
Hence:
\[
\frac{\al}{2} \int_{B_{k}} \frac{|D \un|^2}{(2+|\un|)^{2(\te+\g)}} \le \frac{1}{2\al}\int_{B_{k}}|E|^{2}+\int_{B_{k}} |f|+C,
\]
Taking the sum from $k=0$ to $k=\infty$:
\[
\iom \frac{|D \un|^2}{(2+|\un|)^{2(\te+\g)}} \le C\left(\iom |E|^{2}+\iom |f|+1\right),
\]
For any $q<2$, by Holder's inequality using $\frac 2 q$ and $(\frac{2}{q})'$:
\[
\int_{\Omega} |D(2+|u_{n}|)|^{q}=\int_{\Omega} \frac{(2+|\un|)^{q(\te+\g)}}{(2+|\un|)^{q(\te+\g)}}|Du_{n}|^{q}\le C \left( \iom (2+|\un|)^{\frac{2q(\te+\g)}{2-q}}\right)^{\frac{2-q}{2}}
\]
We choose $q$ such that $\frac{2q(\te+\g)}{2-q}=q^{*}$ , which is to say
\[
q=\frac{2N(1-(\te+\g))}{N-2(\te+\g)}
\]
As before, $\un$ is bounded in $\textbf{W}_{0}^{1,q}$ and up to subsequence $\un\wc u$. Using exactly the same arguments of the proof of theorem \ref{bounds}, we conclude that $u\in\textbf{W}_{0}^{1,q}$ is a distributional solution.
\end{proof}
\vspace{0.1in}
\begin{remark}
The condition $r>\frac{q^{*}}{\g}$ is needed for Holder's inequality when proving that $\int |\un|^{\g}|E|$ is equi-integrable.
\end{remark}
\section{The presence of lower order term}
In this last section we consider the effects on the existence and regularity of the presence of a lower order term in problem \eqref{main}. More precisely, we consider
\begin{equation} \label{main2}\tag{L}
\begin{cases}
 -\Div(a(x,u)Du) + u = -\Div\left(u^{\g}E(x)\right)+f(x) \qquad & \mbox{in } \Omega,\\
u (x) = 0 & \mbox{on }  \partial \Omega,
\end{cases}
\end{equation} 
where $\g>0$ and $a(x,s)$ satisfies \eqref{degc}.

Similar to the previous case, for a fixed $n>0$, Schauder’s fixed point theorem can be used to guarantee the existence of weak solution $u_{n}\in\wpp\cap\L\infty$, satisfying
\begin{equation}\label{eq4}
\int_{\Omega} a(x,T_{n}(u_{n}))Du_{n} D \varphi +\iom \un\vf =\int_{\Omega}\left(T_{n}(\un)^{\g}E_{n}\right) D \varphi  + \int_{\Omega} f_{n} \varphi \quad \forall \varphi \in  \wpp.
\end{equation}

We need the following lemma first:
\vspace{0.1in}
\begin{lemmaletter}\label{lemB}
Suppose $\te+2\g<2$, $E\in [\L {2r}]^{N}$ with $r\ge \frac{m}{2-2\g-\te}$, $f\in\L m$, with $m\ge2$. Then:
\[
\iom |\un|^{m} \le  C\left(\iom|E|^{\frac{2m}{2-2\g-\te}}+\iom |E|^{2}+ \iom |f|^{m} \right).
\]
If $m=1$, for any $\te>0,\g>1$ and $r\ge 1$ we have:
\[
\iom |\un| \le \iom |f|.
\]
\end{lemmaletter}
\begin{proof}
If $m=1$, fix $k>0$ and take $\vf=\frac{T_{k}(\un)}{k}$ as a test function. We have, ignoring the first positive term:
\[
 \iom \un\frac{T_{k}(\un)}{k} \le k^{\g-1}\int_{\Omega}|E||D\un| + \int_{\Omega} |f|
\]
Taking the limit $k\to 0$ and using Fatou's lemma:
\[
\iom |\un| \le\int_{\Omega} |f|
\]
Fix $\la > 1$, take $\vf=|1+|\un||^{\la-2}(1+|\un|)\mathrm{sgn}(\un)$ as a test function to obtain:
\[
\al (\la-1)\iom(1+|\un|)^{\la-2-\te}|D\un|^{2}+\iom |\un|^{\la} \le C \iom (1+|\un|)^{\la-2+\g}|E||D\un| + \int_{\Omega} |f||\un|^{\la-1}
\]
After using Young's inequality, that becomes:
\[
\iom |\un|^{\la} \le C \iom (1+|\un|)^{2[(\la-2+\g)-\frac{\la-2-\te}{2}]}|E|^{2} + \int_{\Omega} |f||\un|^{\la-1}
\]
Simplifying:
\[
\iom |\un|^{\la} \le \frac{1}{4} \iom |\un|^{2r'[(\la-2+\g)-\frac{\la-2-\te}{2}]}+ C\left(\iom|E|^{2r}+\iom |E|^{2}+ \iom |f|^{m} \right) + \frac{1}{4}\iom |\un|^{m'(\la-1)}
\]
Choosing $\la=m$ and $r'=\frac{m}{m-2+2\g+\te}$ we obtain:
\begin{equation}\label{eq5}
\frac{1}{2}\iom |\un|^{m} \le  C\left(\iom|E|^{\frac{2m}{2-2\g-\te}}+\iom |E|^{2}+ \iom |f|^{m} \right) 
\end{equation}
\end{proof}
\subsection{$m,r$ sufficiently large}
As we shall see in the next theorem, the presence of a low order term increase the regularity of solutions. This fact was already noticed in some cases when $\te=0,\g>1$, see \cite{boc24}. Here we extend this analysis to the case $\te>1$.
\vspace{0.1in}
\begin{theorem} \label{bounds4}
Suppose $\te>1$ and $\te+2\g<2$, $f\in\L m$ with $m>\frac{\te N}{2}$, $E\in [\L {p}]^{N}$ such that $p>\frac{Nm}{m-2(\g+\te-1)}$. Then the Dirichlet problem \eqref{main2} has bounded weak solution $u\in\wpp\cap \L \infty$.
\end{theorem}
\begin{proof}
Set $A_{k}=\{(1+|\un|)^{\te-1} > (1+k)^{\te-1}\}=\{ |\un| > k\}$ and take 
\[\vf=\frac{1}{\te-1}G_{(1+k)^{\te-1}}((1+|\un|)^{\te-1})\textrm{sgn}(\un)=:G_{k,n}\textrm{sgn}(\un)
\]
as a test function in \eqref{eq4}. We have:
\[
\al\ioa \frac{|Du_{n}|^{2}}{(1+|\un|)^2} +\ioa |\un| G_{k,n}\le C\ioa(1+|\un|)^{\g+\te-2}|E| |D\un| + \ioa|f||G_{k,n}|
\]
Notice that by Young's inequality:
\[
\begin{split}
\ioa|f|G_{k,n}&\le C_{\e}\ioa |f|^{\te} +\frac{\e}{\te-1}\ioa [(1+|\un|)^{\te-1} - (1+k)^{\te-1}][(1+|\un|)^{\te-1} - (1+k)^{\te-1}]^{\frac 1 {\te-1}}\\ &\le C_{\e}\ioa |f|^{\te} +\e C_{\te}\ioa G_{k,n}|\un|
\end{split}
\]
Taking $\e=\frac 1 {C_{\te}}$ and combining with the equations above we get:
\[
\al\ioa \frac{|Du_{n}|^{2}}{(1+|\un|)^2} \le C\ioa(1+|\un|)^{\g+\te-2}|E| |D\un| + C_{\e}\ioa |f|^{\te}.
\]
Using Young's inequality again:
\[
\frac\al2\ioa \frac{|Du_{n}|^{2}}{(1+|\un|)^2} \le \frac{1}{2\al}\ioa(1+|\un|)^{2(\g+\te-1)}|E|^{2}+ C_{\e}\ioa |f|^{\te},
\]
Choosing $r$ such that $2r'(\g+\te-1)=m$, i.e. $r=\frac{m}{m-2(\g+\te-1)}$, using lemma \ref{lemB} and Holder's inequality we obtain:
\[
\begin{split}
\ioa \left|D\log\left(\frac{1+|\un|}{1+k}\right)\right|^2 &\le C\left(\ioa |E|^{2r}\right)^{\frac 1 r}+ C\norma{f}{m}^{\te}|A_{k}|^{\frac{m-\te}{m}}\\ &\le C\left(\norma{E}{p}^{2}|A_{k}|^{\frac{p-2r}{p}}+\norma{f}{m}^{\te}|A_{k}|^{\frac{m-\te}{m}}\right)
\end{split}
\]
Finally, Sobolev's inequality give us:
\[
\left(\ioa |\log(1+|\un|)-\log(1+k)|^{2^{*}}\right)^{\frac 2 {2^{*}}}\le C\left(\norma{E}{p}^{2}|A_{k}|^{\frac{p-2r}{p}}+\norma{f}{m}^{\te}|A_{k}|^{\frac{m-\te}{m}}\right)\]
We conclude that:
\[
\ioa |\log(1+|\un|)-\log(1+k)|^{2^{*}}\le C\left(\norma{E}{p}^{2}|A_{k}|^{\frac{2^{*}(p-2r)}{2p}}+\norma{f}{m}^{\te}|A_{k}|^{\frac{2^{*}(m-\te)}{2m}}\right)
\]
Since $m>\frac{\te N}{2},p>rN$ implies ${\frac{2^{*}(m-\te)}{2m}},\frac{2^{*}(p-2r)}{2p} >1$, lemma \ref{lemA} gives $\norma{\log(1+|\un|)}{\infty}\le C$ and consequently:
\[
\norma{\un}{\infty}\le C
\]
It suffices now to choose any $n>C$ such that $T_{n}(\un) = \un$, for this particular $n$, $\un$ is a bounded weak solution of problem \eqref{main2}.
\end{proof}
\subsection{$m\ge \te+2$}
In the next case we slightly weaken the summability of the source $f(x)$, the cost of this is the existence of a distributional solution only, instead of a bounded weak solution.
\vspace{0.1in}
\begin{theorem} 
Suppose $\te+2\g<2$, $f\in\L m$ with $m\ge\te+2$, and $E\in [\L {2r}]^{N}$ with $r\ge \max\left(\frac{m}{m-2(\g+\te)},\frac{2^{*}}{2^{*}-\g}\right)$. Then the Dirichlet problem \eqref{main2} has distributional solution $u\in\wpp\cap \L m$.
\end{theorem}
\begin{proof}
Take $\vf=[(1+|\un|)^{\te+1}-1]\mathrm{sgn}(\un)$ as a test function to obtain:
\[
\al \iom |D\un|^{2}\le C \iom (1+|\un|)^{\te+\g}|E||D\un| + C\int_{\Omega} |f||\un|^{\te+1}.
\]
Using Young's inequality and simplifying we have:
\[
\iom |D\un|^{2} \le C \iom (1+|\un|)^{2(\te+\g)}|E|^{2} + C\norma{f}{m}\norma{|\un|^{\te+1}}{m'}
\]
Since $\te+2\le m$ and $r\ge\frac{m}{m-2(\g+\te)}$ we can use lemma \ref{lemB}, which gives:
\[
\iom |D\un|^{2} \le C
\]
Hence, $\un\wc u$ up to a subsequence, using the same reasoning as the proof of theorem \ref{bounds2} we can easily see that $u$ is a distributional solution.
\end{proof}
\subsection{$2 \le m< \te+2$}
In our last result we weaken even more the summability of the source term, yet we are still able to obtain distributional solutions.
\vspace{0.1in}
\begin{theorem} 
Suppose $\te+2\g<2$, $f\in\L m$ with $2\le m<\te+2$, and $E\in [\L {2r}]^{N}$ with $r\ge \max\left(\frac{m}{2-2\g-\te},\frac{m}{6+\te-2m-2\g}\right)$. Then the Dirichlet problem \eqref{main2} has distributional solution $u\in\textbf{W}_{0}^{1,\frac{2m}{\te+2}}\cap\L m$.
\end{theorem}
\begin{proof}
Consider $\vf= [(1+|\un|)^{m-1}-1]\textrm{sgn}(\un)$ as a test function. We have
\[
\iom \frac{|D \un|^2}{(1+|\un|)^{\te-m+2}}\le C\left(\iom (1+|\un|)^{\g+m-2}|E||D \un|+\iom |f||\un|^{m-1}\right)
\]
Using Young's inequality, the fact that $m<\te+2$, and  lemma \ref{lemB} again, we have:
\[
\begin{split}
\iom \frac{|D \un|^2}{(1+|\un|)^{\te-m+2}}&\le C\left(\iom (1+|\un|)^{2[\g+m-2+\frac{\te-m+2}{2}]}|E|^{2}+\iom |f||\un|^{m-1}\right)\\ &\le C\left(\norma{E}{2r}^{2}+\norma{f}{m}^{m}\right)
\end{split}
\]
We conclude that:
\[
\iom \frac{|D \un|^2}{(1+|\un|)^{\te-m+2}}\le C
\]
For any $q<2$, by Holder's inequality using $\frac 2 q$ and $(\frac{2}{q})'$:
\[
\int_{\Omega} |Du_{n}|^{q}=\int_{\Omega} \frac{(1+|\un|)^{\frac{q(\te-m+2)}{2}}}{(1+|\un|)^{\frac{q(\te-m+2)}{2}}}|Du_{n}|^{q}\le C \left( \iom (1+|\un|)^{\frac{q(\te-m+2)}{2-q}}\right)^{\frac{2-q}{2}}
\]
Set $q=\frac{2m}{\te+2}$ then $\frac{q(\te-m+2)}{2-q}=m$ and we conclude that 
\[
\int_{\Omega} |Du_{n}|^{\frac{2m}{\te+2}}\le C.
\]
Therefore, $\un\wc u$ up to a subsequence and as before, $u$ is a distributional solution.
\end{proof}
\section{Concluding remarks and open questions}
Notice that we have assumed $N\ge3$ in this manuscript due to some estimates failing when $N=2$. It would be interesting to see if the arguments presented here can be adapted to include similar results in the plane as well. Hence it's reasonable to ask the following question:

\vspace{0.1in}
\textit{What are the equivalent results of the ones presented here in 2 dimensions?}
\vspace{0.1in}

In this work the assumption $\g>0$ was heavily used, so it would be interesting to see the equivalent results, if any, in the case $\g<0$. Notice in this case the nonlinearity would compete with the degeneracy but this time  also being a singularity so it's possible that no bounded solutions exists and if they do it's possible that some smallness condition will be required contrary to the case described here in theorem \ref{bounds}. We ask the following:

\vspace{0.1in}
\textit{Is it still possible to obtain finite energy solutions if $\g<0$ without smallness condition on the source or vector field $E$? }
\vspace{0.1in}

We can increase the level of difficulty of the Dirichlet problem \eqref{main2} if instead of adding $u(x)$, we add $g(u)$ for some real valued function $g(s)$ with reasonable growth. It would be interesting to see if one can obtain Ambrosetti–Prodi type results in this case. More precisely, consider the problem:
\begin{equation} 
\begin{cases}
 -\Div(a(x,u)Du) + g(u) = -\Div\left(u^{\g}E(x)\right)+f(x) \qquad & \mbox{in } \Omega,\\
u (x) = 0 & \mbox{on }  \partial \Omega,
\end{cases}
\end{equation} 
\vspace{0.1in}
\textit{Is it possible to find a function $g(s)$ Lipschitz with $g(0)=0$ such that for any given source $f(x)$ only one of the following three options are possible: the above system has no solution, one solution, or two solutions.}
\vspace{0.1in}
\bibliography{sn-article}


\begin{thebibliography}{11}
\ifx \bisbn   \undefined \def \bisbn  #1{ISBN #1}\fi
\ifx \binits  \undefined \def \binits#1{#1}\fi
\ifx \bauthor  \undefined \def \bauthor#1{#1}\fi
\ifx \batitle  \undefined \def \batitle#1{#1}\fi
\ifx \bjtitle  \undefined \def \bjtitle#1{#1}\fi
\ifx \bvolume  \undefined \def \bvolume#1{\textbf{#1}}\fi
\ifx \byear  \undefined \def \byear#1{#1}\fi
\ifx \bissue  \undefined \def \bissue#1{#1}\fi
\ifx \bfpage  \undefined \def \bfpage#1{#1}\fi
\ifx \blpage  \undefined \def \blpage #1{#1}\fi
\ifx \burl  \undefined \def \burl#1{\textsf{#1}}\fi
\ifx \doiurl  \undefined \def \doiurl#1{\url{https://doi.org/#1}}\fi
\ifx \betal  \undefined \def \betal{\textit{et al.}}\fi
\ifx \binstitute  \undefined \def \binstitute#1{#1}\fi
\ifx \binstitutionaled  \undefined \def \binstitutionaled#1{#1}\fi
\ifx \bctitle  \undefined \def \bctitle#1{#1}\fi
\ifx \beditor  \undefined \def \beditor#1{#1}\fi
\ifx \bpublisher  \undefined \def \bpublisher#1{#1}\fi
\ifx \bbtitle  \undefined \def \bbtitle#1{#1}\fi
\ifx \bedition  \undefined \def \bedition#1{#1}\fi
\ifx \bseriesno  \undefined \def \bseriesno#1{#1}\fi
\ifx \blocation  \undefined \def \blocation#1{#1}\fi
\ifx \bsertitle  \undefined \def \bsertitle#1{#1}\fi
\ifx \bsnm \undefined \def \bsnm#1{#1}\fi
\ifx \bsuffix \undefined \def \bsuffix#1{#1}\fi
\ifx \bparticle \undefined \def \bparticle#1{#1}\fi
\ifx \barticle \undefined \def \barticle#1{#1}\fi
\bibcommenthead
\ifx \bconfdate \undefined \def \bconfdate #1{#1}\fi
\ifx \botherref \undefined \def \botherref #1{#1}\fi
\ifx \url \undefined \def \url#1{\textsf{#1}}\fi
\ifx \bchapter \undefined \def \bchapter#1{#1}\fi
\ifx \bbook \undefined \def \bbook#1{#1}\fi
\ifx \bcomment \undefined \def \bcomment#1{#1}\fi
\ifx \oauthor \undefined \def \oauthor#1{#1}\fi
\ifx \citeauthoryear \undefined \def \citeauthoryear#1{#1}\fi
\ifx \endbibitem  \undefined \def \endbibitem {}\fi
\ifx \bconflocation  \undefined \def \bconflocation#1{#1}\fi
\ifx \arxivurl  \undefined \def \arxivurl#1{\textsf{#1}}\fi
\csname PreBibitemsHook\endcsname

\bibitem[\protect\citeauthoryear{Alvino et~al.}{2003}]{boc03}
\begin{barticle}
\bauthor{\bsnm{Alvino}, \binits{A.}},
\bauthor{\bsnm{Boccardo}, \binits{L.}},
\bauthor{\bsnm{Ferone}, \binits{V.}},
\bauthor{\bsnm{Orsina}, \binits{L.}},
\bauthor{\bsnm{Trombetti}, \binits{G.}}:
\batitle{Existence results for nonlinear elliptic equations with degenerate
  coercivity}.
\bjtitle{Annali di Matematica Pura ed Applicata}
\bvolume{182}(\bissue{1}),
\bfpage{53}--\blpage{79}
(\byear{2003})
\doiurl{10.1007/s10231-002-0056-y}
\end{barticle}
\endbibitem

\bibitem[\protect\citeauthoryear{Ambrosetti and Rabinowitz}{1973}]{ambrosetti}
\begin{barticle}
\bauthor{\bsnm{Ambrosetti}, \binits{A.}},
\bauthor{\bsnm{Rabinowitz}, \binits{P.H.}}:
\batitle{Dual variational methods in critical point theory and applications}.
\bjtitle{Journal of Functional Analysis}
\bvolume{14}(\bissue{4}),
\bfpage{349}--\blpage{381}
(\byear{1973})
\doiurl{10.1016/0022-1236(73)90051-7}
\end{barticle}
\endbibitem

\bibitem[\protect\citeauthoryear{Boccardo et~al.}{2024}]{boc24}
\begin{barticle}
\bauthor{\bsnm{Boccardo}, \binits{L.}},
\bauthor{\bsnm{Buccheri}, \binits{S.}},
\bauthor{\bsnm{{Rita Cirmi}}, \binits{G.}}:
\batitle{Elliptic problems with superlinear convection terms}.
\bjtitle{Journal of Differential Equations}
\bvolume{406},
\bfpage{276}--\blpage{301}
(\byear{2024})
\doiurl{10.1016/j.jde.2024.06.014}
\end{barticle}
\endbibitem

\bibitem[\protect\citeauthoryear{Boccardo and Croce}{2013}]{boc13}
\begin{bbook}
\bauthor{\bsnm{Boccardo}, \binits{L.}},
\bauthor{\bsnm{Croce}, \binits{G.}}:
\bbtitle{Elliptic Partial Differential Equations Existence and Regularity of
  Distributional Solutions}.
\bpublisher{De Gruyter},
\blocation{Berlin, Boston}
(\byear{2013}).
\doiurl{10.1515/9783110315424}
\end{bbook}
\endbibitem

\bibitem[\protect\citeauthoryear{Brezis and Kamin}{1992}]{brezis92}
\begin{barticle}
\bauthor{\bsnm{Brezis}, \binits{H.}},
\bauthor{\bsnm{Kamin}, \binits{S.}}:
\batitle{Sublinear elliptic equations in rn}.
\bjtitle{manuscripta mathematica}
\bvolume{74}(\bissue{1}),
\bfpage{87}--\blpage{106}
(\byear{1992})
\doiurl{10.1007/BF02567660}
\end{barticle}
\endbibitem

\bibitem[\protect\citeauthoryear{Berestycki and Lions}{1983}]{blions}
\begin{barticle}
\bauthor{\bsnm{Berestycki}, \binits{H.}},
\bauthor{\bsnm{Lions}, \binits{P.-L.}}:
\batitle{Nonlinear scalar field equations, i existence of a ground state}.
\bjtitle{Archive for Rational Mechanics and Analysis}
\bvolume{82}(\bissue{4}),
\bfpage{313}--\blpage{345}
(\byear{1983})
\doiurl{10.1007/BF00250555}
\end{barticle}
\endbibitem

\bibitem[\protect\citeauthoryear{Boccardo}{2009}]{boc09}
\begin{barticle}
\bauthor{\bsnm{Boccardo}, \binits{L.}}:
\batitle{Some developments on dirichlet problems with discontinuous
  coefficients}.
\bjtitle{Bollettino dell'Unione Matematica Italiana}
\bvolume{2}(\bissue{1}),
\bfpage{285}--\blpage{297}
(\byear{2009})
\end{barticle}
\endbibitem

\bibitem[\protect\citeauthoryear{de~Figueiredo et~al.}{1995}]{djairo95}
\begin{barticle}
\bauthor{\bsnm{Figueiredo}, \binits{D.G.}},
\bauthor{\bsnm{Miyagaki}, \binits{O.H.}},
\bauthor{\bsnm{Ruf}, \binits{B.}}:
\batitle{Elliptic equations in r2 with nonlinearities in the critical growth
  range}.
\bjtitle{Calculus of Variations and Partial Differential Equations}
\bvolume{3}(\bissue{2}),
\bfpage{139}--\blpage{153}
(\byear{1995})
\doiurl{10.1007/BF01205003}
\end{barticle}
\endbibitem

\bibitem[\protect\citeauthoryear{Díaz and Letelier}{1993}]{diaz93}
\begin{barticle}
\bauthor{\bsnm{Díaz}, \binits{G.}},
\bauthor{\bsnm{Letelier}, \binits{R.}}:
\batitle{Explosive solutions of quasilinear elliptic equations: existence and
  uniqueness}.
\bjtitle{Nonlinear Analysis: Theory, Methods \& Applications}
\bvolume{20}(\bissue{2}),
\bfpage{97}--\blpage{125}
(\byear{1993})
\doiurl{10.1016/0362-546X(93)90012-H}
\end{barticle}
\endbibitem

\bibitem[\protect\citeauthoryear{Ding and Ni}{1986}]{ding86}
\begin{barticle}
\bauthor{\bsnm{Ding}, \binits{W.-Y.}},
\bauthor{\bsnm{Ni}, \binits{W.-M.}}:
\batitle{On the existence of positive entire solutions of a semilinear elliptic
  equation}.
\bjtitle{Archive for Rational Mechanics and Analysis}
\bvolume{91}(\bissue{4}),
\bfpage{283}--\blpage{308}
(\byear{1986})
\doiurl{10.1007/BF00282336}
\end{barticle}
\endbibitem

\bibitem[\protect\citeauthoryear{Gilbarg and Trudinger}{2001}]{trudinger}
\begin{bbook}
\bauthor{\bsnm{Gilbarg}, \binits{D.}},
\bauthor{\bsnm{Trudinger}, \binits{N.S.}}:
\bbtitle{Elliptic Partial Differential Equations of Second Order},
\bedition{3}rd edn.
\bsertitle{Classics in mathematics},
p. \bfpage{518}.
\bpublisher{Springer}, \blocation{???}
(\byear{2001}).
\doiurl{10.1007/978-3-642-61798-0}
\end{bbook}
\endbibitem

\end{thebibliography}
\end{document}